\documentclass[hidelinks,12 pt,a4paper]{article}
\usepackage[utf8]{inputenc}
\usepackage[english]{babel}
\usepackage{amsmath,amsthm,
enumerate,xcolor,
 mathrsfs,enumitem,amssymb,
 enumitem,graphicx }
 \usepackage{hyperref}
\usepackage[all,knot,cmtip]{xy}
\usepackage[titletoc,title]{appendix}
 \usepackage{color}
 
\renewcommand{\L}{\mathbb L}
\newcommand{\ot}{\otimes}
\newcommand{\Z}{\mathbb{Z}} 
\newcommand{\Q}{\mathbb{Q}}

\newcommand{\Aaut}{\mathscr Aut}

\DeclareMathOperator{\id}{id} 
 \DeclareMathOperator{\Aut}{Aut}
 \DeclareMathOperator{\End}{End}
\DeclareMathOperator{\map}{map}
\DeclareMathOperator{\coker}{coker}
\DeclareMathOperator{\aut}{aut}

\DeclareMathOperator{\rel}{rel}
\DeclareMathOperator{\Hom}{Hom}

\DeclareMathOperator{\Der}{Der}
\DeclareMathOperator{\GL}{GL}
\DeclareMathOperator{\incl}{incl}

\newcommand{\la}{\langle}
\newcommand{\ra}{\rangle}
\DeclareMathOperator{\Tor}{Tors}

\DeclareMathOperator{\im}{im}

\newcommand{\C}{\mathbb{C}}

\newcommand{\E}{\mathcal E}
\newcommand{\g}{\mathfrak g}
\newcommand{\h}{\mathfrak h}
\newcommand{\M}{\mathcal M}

\newtheorem{thm}{Theorem}[section]
\newtheorem{cor}[thm]{Corollary}
\newtheorem{prop}[thm]{Proposition}

\newtheorem{lemma}[thm]{Lemma} 

\theoremstyle{definition}
\newtheorem{dfn}[thm]{Definition}

\newtheorem{rmk}[thm]{Remark}
\newtheorem{conv}[thm]{Convention}

\newcommand{\Addresses}{{
  \bigskip
  \footnotesize
  
\textsc{Department of Mathematics, Stockholm University, SE-106 91 Stockholm, Sweden}\par\nopagebreak
  \textit{E-mail address:} \texttt{hadrien.espic@gmail.com, basharsaleh1@gmail.com}
}}

\begin{document}
\title{On the  group of homotopy classes of relative homotopy automorphisms} 
\author{Hadrien Espic and Bashar Saleh} 
\date{}
\maketitle
\begin{abstract}
We prove that the  group of homotopy classes of relative homotopy automorphisms of a simply connected finite CW-complex is finitely presented and that the rationalization map from this group to its rational analogue has a finite kernel.
\end{abstract}

\section{Introduction}
Given a cofibration $A\subset X$ of simply connected spaces of the homotopy type of finite CW-complexes, we prove that the group of homotopy classes of relative homotopy automorphisms, denoted by $\pi_0(\aut_A(X))$, is finitely presented. Moreover, we prove that the group homomorphism $\pi_0(\aut_A(X))\to \pi_0(\aut_{A_\Q}(X_\Q))$ induced by rationalization has a finite kernel.

Recall that the space of relative homotopy automorphisms $\aut_A(X)$ is the space of all homotopy automorphisms of $X$ that preserve $A$ pointwise. The homotopy theory of relative homotopy automorphisms and algebraic models for them is a topic of activity (\cite{BM14}, \cite{grey}, \cite{BS19}), especially in the study of homological stability for homotopy automorphisms of connected sums of manifolds.

The result of this paper is, to some extent, an extension of some of the consequences of the classical Sullivan-Wilkerson Theorem \cite[Theorem 10.3 (i)]{sullivan77}, \cite[Theorem B (2)]{wilkerson76} (not affected by the erratum \cite{wilkersonerrata}) stating that $\pi_0(\aut(X_\Q))$ is the group of $\Q$-points of an algebraic group $\mathscr G$ and that $\pi_0(\aut(X))$ is commensurable up to finite kernel with an arithmetic subgroup of $\mathscr G$. Two  important consequences of the Sullivan-Wilkerson theorem are that $\pi_0(\aut(X))$ is finitely presented (since arithmetic groups are finitely presented), and that the finiteness of $\pi_0(\aut(X_\Q))$ implies the finiteness of $\pi_0(\aut(X))$. In this paper we show that these two properties also hold in the relative case.

\begin{thm}\label{thm:AA}
Let  $A\subset X$ be a cofibration of simply connected spaces of the homotopy type of finite CW-complexes. Then $\pi_0(\aut_A(X))$ is finitely presented and the map $\pi_0(\aut_A(X))\to \pi_0(\aut_{A_\Q}(X_\Q))$ has a finite kernel.
\end{thm}
This is proved in Theorem \ref{Thm:finiteKer} and Theorem \ref{thm:finitePresentation}. Based on a preprint version of this paper, Kupers \cite[Proposition 3.4]{kupers22} proves a full generalization of the relative Sullivan-Wilkerson theorem.

Our proof of the finite presentation property, depends on the classical Sullivan-Wilkerson Theorem, which involves some theory of algebraic and arithmetic groups. In order to be able to use the classical Sullivan-Wilkerson Theorem, we prove the following:

\begin{thm}\label{thm:B}
Let $A\subset X$ be a cofibration of simply connected spaces of  the homotopy type of finite CW-complexes. Then $\pi_0(\aut_{A_\Q}(X_\Q))$ is the group of $\Q$-points of a linear algebraic group and the map $\pi_0(\aut_{A_\Q}(X_\Q))\to \pi_0(\aut(X_\Q))$ induced by the inclusion $\aut_{A_\Q}(X_\Q)\hookrightarrow \aut(X_\Q)$ is given by the $\Q$-points of a homomorphism of linear algebraic groups.
\end{thm}
This is proved in Theorem \ref{thm:alggrp} and Theorem \ref{thm:AlgMap}. The algebraicity of the non-relative homotopy automorphisms of a rational space $X_\Q$ is usually proved by showing that the group of homotopy classes of  automorphisms of a minimal Sullivan model for $X_\Q$ is the group of $\Q$-points of an algebraic group. However, this approach is not suitable for modelling relative homotopy automorphisms. Rather, one should consider homotopy classes of relative  automorphisms of so called minimal relative dg Lie algebra models. The authors could not find an explicit treatment of the theory of minimal relative dg Lie algebra models in  the literature, so Section \ref{sec:minimalRelModels} is devoted to that. A dg Lie algebra $(\L(V),d)$ will be abbreviated by $\L(V)$.   In particular we prove the following:

\begin{thm}
 Given a map $f\colon \L(V)\to \g$ of simply connected dg Lie algebras, there exists a minimal relative model $q\colon\L(V\oplus W)\xrightarrow\sim\g$ for $f$ in the following sense:
 \begin{itemize}
 \item[{\normalfont (a)}] $\L(V)$ is a dg subalgebra of $\L(V\oplus W)$ and $f = q\circ \iota$, where $\iota\colon \L(V)\to \L(V\oplus W)$ is the inclusion.
 \item[{\normalfont(b)}] Every quasi-isomorphism $ \L(V\oplus W)\to \L(V\oplus W)$, which restricts to an automorphism of $\L(V)$,  is an automorphism.
 \end{itemize}
\end{thm}This is proved in Theorem \ref{thm:minimalModel} and Theorem \ref{thm:hautOfMinimal}.
Here, $\L(A)$ denotes a quasi-free dg Lie algebra generated by $A$, i.e. a dg Lie algebra whose underlying graded Lie algebra structure is free. 

Lastly, we relate Theorem \ref{thm:AA} to the relative homotopy automorphisms of certain Postnikov stages of finite CW-complexes. The $n$-Postnikov stage of a space $X$ is denoted by $X(n)$.
\begin{lemma}
Let $X$ be a simple finite CW-complex and let $m \geq \dim(X)$. Pick $X(m)$ so that $X \to X(m)$ is a cofibration (and thus $A\to X(m)$ is also a cofibration). Then there is an isomorphism of groups $\pi_0(\aut_A(X(m)))\cong \pi_0(\aut_A(X))$.
\end{lemma}
This is the main result of Section \ref{sec:finite-postnikov}. We conclude from this lemma and  Theorem \ref{thm:AA}  the following:
\begin{cor}\label{cor:kupers}
Let $A \subset X$ be a cofibration of simply connected finite CW-complexes and let $m \geq dim(X)$. Then $\pi_0(\aut_A(X(m)))$ is finitely presented.\end{cor}

\subsubsection*{On the results of Scheerer and Maruyama} In   \cite{scheerer80}, the arithmeticity of the group of homotopy classes of fiber homotopy equivalences is announced. In the article, Scheerer suggests that an Eckmann-Hilton-dualization of his arguments would yield the arithmeticity of the group of homotopy classes of relative homotopy automorphisms. However, Eckmann-Hilton duality is  a non-formal duality and details have to be spelled out. In particular, the Eckmann-Hilton dual of the ``Postnikov tower construction" is the ``homology decomposition construction", which differs in  the following crucial sense; Postnikov towers can be made functorial, while  homology decompositions can not \cite[Proposition 6.13]{baues95}. Since  Scheerer's arguments rely on the functoriality of the Postnikov tower construction, we do not see an immediate way of dualizing all arguments.

Maruyama \cite{maruyama} proves that certain subgroups of $\pi_0(\aut(X))$ are finitely presentable. Since $\pi_0(\aut_A(X))$ is not a subgroup of $\pi_0(\aut(X))$, we can not deduce finite presentability directly from the results of Maruyama. However, Maruyama's results might  provide some information about $\pi_0(\aut_A(X))$ in the following way: Let $i\colon A\subset X$ be a cofibration of nilpotent finite CW-complexes and let $[i]\in [A,X]$ be the homotopy class of $i$. We have that $\pi_0(\aut(X))$ acts on $[A,X]$ by post-composition. Let $\pi_0(\aut(X))_{[i]}$ be the stabilizer group of $[i]$. This group is finitely presented by \cite[2.7 (2)]{maruyama}. A study of the canonical map $\pi_0(\aut_A(X))\to \pi_0(\aut(X))_{[i]}$, might yield insights about  $\pi_0(\aut_A(X))$. 

\subsubsection*{Overview}
In Section \ref{sec:finiteKer}, we prove that the map $\pi_0(\aut_A(X))\to \pi_0(\aut_{A_\Q}(X_\Q))$ induced by rationalizations has a finite kernel provided that $A$ and $X$ are of the homotopy type of finite simply connected CW-complexes. 

In Section \ref{sec:minimalRelModels}, we prove the existence and the properties of minimal dg Lie algebra models for a wide class of maps of dg Lie algebras.

In Section \ref{sec:AlgGrps}, we use the theory of minimal dg Lie models in order to prove that $\pi_0(\aut_{A_\Q}(X_\Q))$ is a linear algebraic group and that the map $\pi_0(\aut_{A_\Q}(X_\Q))\to \pi_0(\aut(X_\Q))$ induced by the inclusion $\aut_{A_\Q}(X_\Q)\hookrightarrow \aut(X_\Q)$ is a homomorphism of linear algebraic groups. 

In Section \ref{sec:finitePres}, we prove that $\pi_0(\aut_A(X))$ is finitely presented using algebraicity of $\pi_0(\aut_{A_\Q}(X_\Q))$ and the map $\pi_0(\aut_{A_\Q}(X_\Q))\to \pi_0(\aut(X_\Q))$.

In Section \ref{sec:finite-postnikov}, we prove that if $X$ is an $n$-dimensional CW-complex then there is an isomorphism of groups $\pi_0(\aut_A(X(n)))\cong \pi_0(\aut_A(X))$ where $X(n)$ is the $n$-Postnikov stage of $X$. This yields Corollary \ref{cor:kupers}.

\subsubsection*{Standing assumptions and notation}

\begin{itemize}
\item A rationalization of a simply connected space $X$ is denoted by $X_\Q$, and a rationalization of a map of simply connected spaces $f\colon X\to Y $ is denoted by 
$f_\Q\colon X_\Q\to Y_\Q$. By \cite{farjoun}, there are functorial and continuous rationalization functors that preserve cofibrations. In particular, given a cofibration $A\subset X$, there is a rationalization functor that induces  a group homomorphism $\pi_0(\aut_A(X))\to \pi_0(\aut_{A_\Q}(X_\Q))$.

\item The notion commensurability in both \cite{sullivan77}  and \cite{wilkerson76} differs from the current common usage of
this term nowadays (see e.g. \cite{K-RW}). We will follow the terminology of \cite{K-RW} and say say that two groups are \textit{commensurable up to finite kernel} if there exists a zigzag of maps
$$G \rightarrow G_1\leftarrow\dots \rightarrow G_n\leftarrow H$$ where each map has finite kernel and the image has finite index. 

We say that two groups $G$ and $H$  are \textit{commensurable} if there are subgroups $G'\subset G$ and $H'\subset H$ of finite index such that $G'\cong H'$.

If $G$ and  $H$ are commensurable up to finite kernel and $G$ is finitely presented then $H$ is also finitely presented.
\end{itemize}

\begin{center}\textbf{Acknowledgements}\end{center}
We are very thankful to Alexander Berglund, Muriel Livernet, Alexander Kupers and the anonymous referee  for their very careful reading of this paper and for pointing out mistakes in previous versions of this paper.
 
We would also like to thank Joana Cirici and Agustí Roig\footnote{Agustí Roig sadly passed away in September 2021.} for suggesting how minimal relative models for a map of operad algebras should be defined. We applied their ideas for the special case of dg Lie algebra maps, which is the topic of Section \ref{sec:minimalRelModels}. Agustí Roig mentioned that the ideas he shared with us have to a large extent already been suggested by Daniel Tanré to him.

We would also like to thank Shun Wakatsuki for several useful discussions.

The second named author was supported by the Knut and Alice Wallenberg Foundation through grant no. 2019.0521. 
 
\section{Finite kernel}\label{sec:finiteKer}
We will briefly recall some of the basic obstruction theory that is covered in many introductory books in algebraic topology (e.g. \cite[§VII.13]{bredonGT}). After that we prove that $\pi_0(\aut_A(X))\to \pi_0(\aut_{A_\Q}(X_\Q))$ has finite kernel.

Let $Y(k)$ denote the $k$-Postnikov stage of $Y$ for $k>0$  and let $Y(0)$ be just be a point which we call the zero Postnikov stage, and let $\rho_k\colon Y\to Y(k)$, $k\geq 0$, denote the structure map for  the $k$-Postnikov stage. 

\begin{prop}\label{prop:Obstr}
Let $A\subset X$ be a cofibration and let $Y$ be a simple space. Let $f,g\colon X\to Y$ be maps that agree on $A$.

Assume that $\rho_kf \simeq_{\rel A} \rho_kg$ for some $k\geq 0$. Then there exists an associated obstruction cohomology class $d(\rho_k f,\rho_k g)\in H^{k+1}(X,A;\pi_{k+1}(Y))$ which satisfies the following properties:
\begin{itemize}
\item[\text{\normalfont (a)}] $\rho_{k+1}f$ and $\rho_{k+1}g$ are homotopic relative to $A$ if and only if $d(\rho_k f,\rho_k g)=0$. 
\item[\text{\normalfont (b)}] The association $(\rho_k f,\rho_k g)\mapsto d(\rho_k f,\rho_k g)$ is natural with respect to maps $Y\to Y'$ and maps of pairs $(X',A')\to (X,A)$.
\item[\text{\normalfont (c)}] Given a third map $h\colon X\to Y$ that agrees with $f$ on $A$ and where $\rho_k h\simeq_{\rel A} \rho_k f$, we get that the obstruction classes satisfy the relation $d(\rho_k f,\rho_k h) = d(\rho_k f,\rho_k g)+d(\rho_k g,\rho_k h)$.
\item[\text{\normalfont (d)}] If $\rho_k f\simeq_{\rel A}\rho_k g$ for all $k$, i.e. all obstruction classes vanish, then $f\simeq_{\rel A} g$.
\end{itemize}
\end{prop}

\begin{rmk}\label{rmk:postnikov} It follows from Proposition \ref{prop:Obstr} (a) and (d) that if $X$ is $n$-dimensional, then it is enough to show that $\rho_nf\simeq_{\rel A}\rho_ng$ in order to deduce that $f\simeq_{\rel A} g$.
\end{rmk}

\begin{thm}\label{Thm:finiteKer}
Given a cofibration $A\subset X$ where $A$ is nilpotent and $X$ is simple, and both are of the homotopy type of finite CW-complexes, the map $\pi_0(\aut_A(X))\to \pi_0(\aut_{A_\Q}(X_\Q))$ has finite kernel.
\end{thm}
\begin{proof}
 Let $\rho_k\colon X\to X(k)$ and $\rho_k^\Q\colon X_\Q\to X_\Q(k)$ be the structure maps for the $k$-Postnikov stages for $X$ and $X_\Q$ respectively, and let $$K = \ker(\pi_0(\aut_A(X))\to \pi_0(\aut_{A_\Q}(X_\Q))).$$
Given an element $f\in K\smallsetminus \{\id_X\}$, we have that  $f\not\simeq_{\rel A} \id_X$  but $f_\Q\simeq_{\rel A_\Q} \id_{X_\Q}$. Since $X(0)=*$, it follows that $\rho_0f = \rho_0\id_X$. Given that 
$f\not\simeq_{\rel A} \id_X$, there must be some $k$, $0<k<\dim(X)$, for which $\rho_{k+1} f\not\simeq_{\rel A} \rho_{k+1}\id_X$ but for which $\rho_{k}f\simeq_{\rel A} \rho_{k}\id_X$. We associate to $f$ the non-zero class $d(\rho_kf,\rho_k)\in H^{k+1}(X,A;\pi_{k+1}(X))$. However, since $f_\Q\simeq_{\rel A_\Q} \id_{X_\Q}$, it follows that $d(\rho_k^\Q f_\Q,\rho^\Q_k)=0$, and due to the naturality of these obstructions (see Proposition \ref{prop:Obstr} (a)), it follows that $d(\rho_kf,\rho_k)\in \Tor(H^{k+1}(X,A;\pi_{k+1}(X))$, where $\Tor$ denotes the torsion part. This gives us a well defined map 
$$\alpha\colon K\smallsetminus\{\id_X\}\to \bigsqcup_{i=1}^{\dim(X)} \Tor(H^i(X,A;\pi_i(X))).$$
Now assume to get a contradiction that $K$ is infinite. This means that $\alpha$ is a map from an infinite set to a finite set, and we may pick some obstruction class $$c\in \Tor(H^{t+1}(X,A;\pi_{t+1}(X)))$$ where $\alpha^{-1}(c)$ is infinite. Pick some  $g\in \alpha^{-1}(c)$. For  any other class $ h\in\alpha^{-1}(c)$ distinct from $g$ we have that 
$d(\rho_t g,\rho_t) = d(\rho_t h,\rho_t)$ and thus by Proposition \ref{prop:Obstr} (c), $\rho_{t+1} g\simeq_{\rel A} \rho_{t+1} h$. Since $g\not\simeq_{\rel A} h$ there must be some $k'$, $t<k'<\dim(X)$, for which $\rho_{k'+1}g\not\simeq_{\rel A}\rho_{k'+1}h$ but where $\rho_{k'}g\simeq_{\rel A}\rho_{k'}h$. This gives a nontrivial obstruction class 
$d(\rho_{k'}g,\rho_{k'}h)\in H^{k'+1}(X,A;\pi_{k'+1}(X))$, and as before we get a map
$$
\alpha'\colon \alpha^{-1}(c)\smallsetminus\{g\}\to \bigsqcup_{i'=t+1}^{\dim(X)} \Tor(H^{i'}(X,A;\pi_{i'}(X))).
$$
After a finite number of iterations of this procedure we will reach a function
$$\mathscr A\colon S\to \Tor(H^{d}(X,A;\pi_{d}(X))),$$ 
where $d=\dim(X)$ and $S$ is an infinite set. Assume that $p,q\in \mathscr A^{-1}(x)$. Then it would mean that $\rho_d p \simeq_{\rel A} \rho_d q$, and thus $p\simeq_{\rel A} q$ (see Remark \ref{rmk:postnikov}), i.e. $p$ and $q$ are identified in $\pi_0(\aut_A(X))$. Thus $\mathscr A$ is injective, which is impossible due to the infinite cardinality of $S$. From the contradiction we conclude that $K$ is finite.
\end{proof}

\section{On minimal relative  dg Lie algebras}\label{sec:minimalRelModels}
This section is inspired by \cite{CiriciRoig}, in which minimal models for algebras over algebraic operads are constructed. We adopt the general framework and generalize some of the theory to  relative dg Lie algebras. We have chosen to only focus on dg Lie algebras, but we believe that the results of this section can be obtained for algebras over  algebraic operads in general.

 The notion of a minimal relative dg Lie algebras seems to not have attained much of attention in the literature. In \cite[§ 22 (f)]{felixrht}, relative dg Lie algebras are discussed, but not in the minimal setting.

\begin{conv} \begin{itemize}
    \item In this section we enumerate vector spaces in a collection by upper indices;  $\{V^1,V^2,\dots\}$. The lower indices are reserved for denoting the homological degree. 
    \item A quasi free dg Lie algebra $(\L(V),d)$ will be abbreviated by $\L(V)$. 
\end{itemize} 
\end{conv}

We start by recalling the notion of KS-extensions of a quasi-free dg Lie algebra. 

\begin{dfn}
A \textit{KS-extension} of a quasi-free dg Lie algebra $\L(V)$ is a quasi-free dg Lie algebra $\L(V\oplus U)$, that contains $\L(V)$ as a dg Lie subalgebra and in which $d(U)\subset Z_*(\L(V))$, where $Z_*(\L(V))$ denotes the cycles of $\L(V)$.
\end{dfn}

\begin{dfn}\label{dfn:RelLieAlg}
A  \textit{relative dg Lie algebra} with \textit{base} $\L(V)$ is a colimit $\mathcal M= \bigcup_{i\geq 0} \mathcal M[n]$ of a sequence
$$
\mathcal M[0] = \L(V)\to \mathcal M[1] = \L(V\oplus W^1)\to\mathcal M[2]= \L(V\oplus W^1\oplus W^2)\to \cdots
$$
of  KS-extensions starting from  the base Lie algebra $\mathcal M[0] =\L(V)$. We have that $\mathcal M=\L(V\oplus W)$ where $W= \bigoplus_{i=1}^\infty W^i$.
\end{dfn}

\begin{dfn}\label{dfn:MinimalRelLieAlg} Let $\pi_W\colon \L(V\oplus W)\to W$ denote the canonical projection. A relative dg Lie algebra $\L(V\oplus W)$ with base $\L(V)$ is called  \textit{minimal} if $\pi_W\circ d(W)=0$
\end{dfn}

\begin{dfn}
 Given a map $f\colon \L(V)\to \g$ of dg Lie algebras, a \textit{(minimal)  Lie model for $f$} is a (minimal) relative Lie algebra $\L(V\oplus W)$ with base $\L(V)$ together with a quasi-isomorphism
 $q\colon\L(V\oplus W)\to \g$, which satisfies the equality $f = q\circ \iota$, in which $\iota\colon \L(V)\to \L(V\oplus W)$ is the canonical inclusion.
\end{dfn}

\begin{dfn}
A dgl $\g$ is said to be \textit{simply connected} if it is concentrated in positive degrees; $\g = \g_{\geq 1}$.
\end{dfn}

We prove the existence of minimal relative models for a class of dg Lie algebra maps.

\begin{thm}\label{thm:minimalModel}
 Given a map $f\colon \L(V)\to \g$ of simply connected dg Lie algebras, there exists a minimal relative model for $f$.
\end{thm}

\begin{proof} 
 We will construct a sequence of minimal KS-extensions $$
\mathcal M[0] = \L(V)\to \mathcal M[1] = \L(V\oplus W^1)\to\mathcal M[2]= \L(V\oplus W^1\oplus W^2)\to \cdots
$$
where each $\mathcal M[n]$ is equipped with a map $q_n\colon \mathcal M[n]\to \g$ such that $q_0 = f$ and
\begin{itemize}
\item[(a)] $q_n$ is an extension of $q_{n-1}$,
\item[(b)] $H_i(q_n)$ is an isomorphism for $i\leq n$,
\item[(c)] $W^n$ is decomposed as $W^n = A^n\oplus B^n$, where $A^n$ is of homogeneous degree $n$ and $B^n$ is of homogeneous degree $n+1$,
\item[(d)] $d(A^n) = 0$,
\item[(e)] $d(B^n)\subseteq \L(V\oplus W(n-2)\oplus A^{n-1})$, where $W(n-2)= \bigoplus_{i=1}^{n-2}W^i$ if $n\geq 3$ and where $W(-2)=W(-1)=W(0)=0$.
\item[(f)] $d|_{B^n}$ is injective and $d(B^n)\cap B_n(\mathcal M[n-1])=\{0\}$, where $B_n(\mathcal M[n-1])$ denotes the $n$-boundaries of $\mathcal M[n-1]$. 
\end{itemize}
Conditions (a) and (b) gives that the colimit $\mathcal M = \bigcup \mathcal M[i]$ is a relative model for $f$. Conditions (c)-(e) gives that $\mathcal M$ is also minimal (note that $d(B^n)$ is homogeneous of degree $n$ and belongs to a free Lie algebra whose generators are either elements of $V$ or elements of maximal degree $n-1$, which by degree reasons gives the relative minimality properties). Condition (f) is a technical condition needed for the inductive construction of the minimal model as described below.

The conditions above are trivially satisfied for $n=0$ (in which we set $\mathcal M[-1] = \mathcal M[0]$, $q_{-1}=q_0=f$ and $W^0=0$), since $\L(V)$ and $\g$ are concentrated in strictly positive degrees. Assume inductively that (a)-(f) are satisfied for $n\leq k-1$ for some fixed $k$. We want to show that this implies that (a)-(f) are satisfied for $n=k$.

Let 
$$A^k= \coker(H_k(q_{k-1})\colon H_k(\mathcal M[k-1])\to H_k(\g)),$$
and let $a_k\colon A^k\to \g_k$ be given by the  composition
$$
a_k\colon A^k\xrightarrow\sigma H_k(\g)\xrightarrow\tau Z_k(\g)\hookrightarrow\g_k
$$ where $\sigma$ and $\tau$ are some sections to the projections $H_k(\g)\twoheadrightarrow A^k$ and $Z_k(\g)\twoheadrightarrow H_k(\g)$ respectively. Let $\mathcal M'[k]=\L(V\oplus W(k-1)\oplus A^k)$ be an extension of $\mathcal M[k-1]$, where $d|_{A^k}=0$. Now we have an extension $ q'_k\colon\mathcal M'[k]\to \g$, of $q_{k-1}$, where $ q'_k|_{A^k}=a_k$. We have that $H_i(q'_k)$ is an isomorphism for $i<k$ and an epimorphism for $i=k$ (by construction of $a_k\colon A^k\to \g$). 

Let $B^k = s(\ker(H_k(q_{k-1})\colon H_k(\mathcal M[k-1])\to H_k(\g)))$ be concentrated in degree $k+1$ and let $\mathcal M[k] = \L(V\oplus W(k-1)\oplus A^k\oplus B^k)$ and let $d|_{B^k}$ be given by the following composition:
$$d|_{B^k}\colon B^k\hookrightarrow sH_k(\mathcal M[k-1])\xrightarrow{s^{-1}}H_k(\mathcal M[k-1])\xrightarrow\nu Z_k(\mathcal M[k-1])$$
where $\nu$ is a section of the projection $Z_k(\mathcal{M}[k-1])\twoheadrightarrow H_k(\mathcal M[k-1])$. 

It follows from the way we defined $d|_{B^k}$ that $\im(q_{k-1} \circ d|_{B^k})\subset B_k(\g)$ (where $B_k(\g)$ denotes the boundaries in $\g$ of degree $k$). Let $\zeta\colon B_k(\g)\to g_{k+1}$ be a section of the differential restricted to $\g_{k+1}$ and let $q_k$ be the extension of $q'_k$ where $q_k|_{B^k} = \zeta\circ q_{k-1} \circ d|_{B^k}$.

This process kills the kernel in homological degree $k$ of the epimorphism $H_k(q'_k)$, which makes $H_k(q_k)$ into a monomorphism and thus an isomorphism. Hence, the conditions (a)-(d) and (f) are satisfied  for $n=k$, by construction. 

Due to degree reasons we have that
$$\mathcal M[k-1]_k =\L(V\oplus W(k-2)\oplus A^{k-1})_k\oplus B^{k-1}.$$ 
Thus for a given $b\in B^k$, there is some $a\in \L(V\oplus W(k-2)\oplus A^{k-1})_k$ and some $b'\in B^{k-1}$, such that $db = a +b'$. 
Consequently, $da+db'=0$. Since we are assuming that (f) holds for $n=k-1$, it follows that $b'=0$, since if $b'\neq 0$, then by (f) we have that $db'\neq0$ and $db'\neq -da$, making the equality $da+db'=0$ impossible. From this, we conclude that $db= a\in \L(V\oplus W(k-2)\oplus A^{k-1})_k$, giving us property (e).
\end{proof}

\begin{thm}[\text{cf. \cite[Theorem 14.11]{felixrht}}]\label{thm:hautOfMinimal}
Given a simply connected minimal relative Lie algebra $\L(V\oplus W)$ with base $\L(V)$, let $f\colon \L(V\oplus W)\to \L(V\oplus W)$ be a quasi-isomorphism. Suppose that $f$ restricts to an automorphism of $\L(V)$, then $f$ is an automorphism.
\end{thm}

\begin{proof}
Let $W_{\leq k}\subseteq W$ be the subspace of elements of at most degree $k$ and let $\mathcal M\la k\ra= \L(V\oplus W_{\leq k})$. Note that $f(\mathcal M\la k\ra)\subseteq\mathcal M\la k\ra$ by degree reasons, hence for every $k\geq 0$, we have that $f$  defines a chain map $\mathcal M\la k\ra\to\mathcal M\la k\ra$.  We start by showing that for every $k\geq0$ there exists a chain map $g_k\colon \mathcal M\la k\ra\to\mathcal M\la k\ra$ such that $fg_k= \id$. In particular $g_k$ is a monomorphism. For $k=0$, this is trivially true, since we are assuming that the restriction of $f$ to $\L(V)$ is an automorphism, so we can set $g_0 = (f|_{\L(V)})^{-1}$. Assume inductively that the statement is true for some $k\geq 0$.  
The map $f$ induces a map of short exact sequences of chain complexes
$$
\xymatrix{0\ar[r]&g_k(\mathcal M\la k\ra)\ar[r]\ar[d]_{f|_{g_k(\mathcal M\la k\ra)}}& \mathcal M\ar[d]_f\ar[r]& \mathcal M/g_k(\mathcal M\la k\ra)\ar[d]_{\bar f}\ar[r]&0
\\
0\ar[r]&\mathcal M\la k\ra\ar[r]& \mathcal M\ar[r]& \mathcal M/\mathcal M\la k\ra\ar[r]&0.}
$$
 The condition $fg_k=\id$ gives that $f|_{g_k(\mathcal M\la k\ra)}$ is surjective.  Since $g_k\circ f|_{g_k(\mathcal M\la k\ra)} = \id_{g_k(\mathcal M\la k\ra)}$, it follows that $f|_{g_k(\mathcal M\la k\ra)}$ is also injective. Hence  $f|_{g_k(\mathcal M\la k\ra)}$ is an isomorphism (and thus also a quasi-isomorphism).

Since $f$ and  $f|_{g_k(\mathcal M\la k\ra)}$ are quasi-isomorphisms, it follows that $\bar f$ is a quasi-isomorphism.
 We have that $(\mathcal M/\mathcal M\la k\ra)_{k+1}=W_{k+1}$. By degree reasons, all elements of $W_{k+1}$ are cycles in $\mathcal M/\mathcal M\la k \ra$, i.e.  $Z_{k+1}(\mathcal M/\mathcal M\la k \ra)= W_{k+1}$. By minimality we have that $B_{k+1}(\mathcal M/\mathcal M\la k \ra)=0$, so it follows that 
$$H_{k+1}(\mathcal M/\mathcal M\la k\ra) =Z_{k+1}(\mathcal M/\mathcal M\la k\ra) = W_{k+1}.$$
Since $\bar f$  is a quasi-isomorphism, we get an epimorphism
$$Z_{k+1}(\mathcal M/g_k(\mathcal M\la k\ra))\twoheadrightarrow W_{k+1}.$$

Let $\xi\colon W_{k+1}\to \M_{k+1}$ be given by the following composition
$$
\xi\colon W_{k+1}\to Z_{k+1}(\mathcal M/g_k(\mathcal M\la k\ra))\hookrightarrow(\mathcal M/g_k(\mathcal M\la k\ra))_{k+1}\to \M_{k+1}
$$ 
where the first and the last maps are sections to the epimorphisms $Z_{k+1}(\mathcal M/g_k(\mathcal M\la k\ra))\twoheadrightarrow W_{k+1}$ and $\M_{k+1}\twoheadrightarrow (\M/g_k(\M\la k\ra))_{k+1}$ respectively. Fix a  basis $w_1,\dots,w_a$ for $W_{k+1}$. One deduces that $f(\xi(w_i)) = w_i + a_i$ for some $a_i\in \M\la k \ra$, and  that $d(\xi(w_i)) = g_k(a'_i)$ for some $a'_i\in \M\la k\ra$ (since $\xi(w_i)$ is a cycle in $\M/g_k(\M\la k\ra)$). Let $g_{k+1}$ be an extension of $g_k$ to $\M\la k+1\ra$ where $g_{k+1}(w_{i}) = \xi(w_i)-g_k(a_i)$. In particular we have that
$$f(g_{k+1}(w_i))= f(\xi(w_i)-g_k(a_i)) = w_i.$$

Using the equality $w_i =f(\xi(w_i)-g_k(a_i))$ (from above) one deduces that 
$$ g_{k+1}(dw_i) =g_k(a_i-da_i') = d(g_{k+1}(w_i)),$$
which shows that $g_{k+1}$ is a chain map. This concludes the inductive step. Now let $g = \lim g_k$ and we have that $fg = \id_\M$ which proves that $g$ is an injective quasi-isomorphism. Applying the same arguments to $g$ gives that there is a map $h$ such that $gh = \id_\M$ which proves that $g$ is surjective. Thus $g$ is an isomorphism and $f$ is its inverse.
\end{proof}

\section{Algebraic groups and relative homotopy automorphisms }\label{sec:AlgGrps}
The rational homotopy category of simply connected pointed spaces is modelled by the homotopy category of simply connected dg Lie algebras (\cite{quillen69}). In particular, Quillen constructed a functor $\lambda\colon \mathrm{Top}_*^\text{1-conn}\to \text{DGL}_\Q^\text{1-conn}$, from the category of simply connected pointed spaces to the category of simply connected dg Lie algebras  (a dg Lie algebra is called simply connected if it is concentrated in strictly positive degrees), that induces the equivalence of homotopy categories mentioned above.

Moreover, the category of dg Lie algebras is enriched over simplicial sets; let $L$ and $\Pi$ be two dg Lie algebras, then the simplicial mapping space is given by
$$
\mathrm{Map}(L,\Pi)_\bullet = \Hom(L,\Pi\ot\Omega_\bullet)
$$
where $\Omega_\bullet$ denote the simplicial commutative dg algebra in which $$\Omega_n = \Lambda(t_0,\dots,t_n,dt_1,\dots,dt_n)/(1-\sum t_i,\sum dt_i)$$ is the Sullivan-de Rham algebra of polynomial differential forms on the $n$-simplex (see \cite[§ 10 (c)]{felixrht} for details). We recall that the tensor product $\Pi\ot \Omega$ of a dg Lie algebra $\Pi$ with a commutative dg algebra $\Omega$ is again a dg Lie algebra, where $[\ell_1\ot c_1, \ell_2\ot c_2] = (-1)^{|c_1||l_2|}[\ell_1,\ell_2]\ot c_1c_2$.

\begin{lemma}\label{mk}
Let $X$ and $Y$ be simply connected spaces where $Y$ is of finite type, and let $\L_X$ and $\L_Y$ be quasi-free models for $X$ and $Y$ respectively. Then there is a weak equivalence 
$$
\mathrm{Map}(\L_X,\L_Y)_\bullet \simeq \mathrm{Map}_*(X_\Q,Y_\Q).
$$
\end{lemma}
\begin{proof}
This follows from \cite[Theorem 2.2.5]{hinich01} and \cite[Corollary 3.17]{BS19} (see also \cite[§ 3.6]{berglund17}).
\end{proof}

 We note that $\Omega_1$ is isomorphic to $\Lambda(t,dt)$ as an algebra. From this one can deduce the following notion of homotopy:

\begin{dfn}
Let $\Lambda(t,dt)$ denote the free graded commutative algebra on $t$ and $dt$ where $|t|=0$ and $|dt|=1$, endowed by a differential that takes $t$ to $dt$. For any $s\in\Q$, let $ev_s\colon \Lambda(t,dt)\to \Q$ be the unique commutative dg algebra map given by $ev_s(t) = s$.
 
Two dg Lie algebra maps  $\varphi,\psi\colon L\to \Pi$ are \textit{homotopic} if and only if there exists a dg Lie algebra map $h\colon L\to \Pi\ot \Lambda(t,dt)$ such that $(\id\ot ev_0)\circ h = \varphi$ and $(\id\ot ev_1)\circ h= \psi$.
\end{dfn}

\begin{rmk}
The notion of homotopic maps make sense over any algebra $R$ over $\Q$; just replace $\Lambda(t,dt)$ by $\Lambda_R(t,dt)$.
\end{rmk}

\begin{dfn}
A dg Lie algebra morphism between quasi-free dg Lie algebras $\varphi\colon \L(V)\to \L(V')$ is called a \textit{free map} if it is injective and if $\varphi(V)\subseteq V'$. 
\end{dfn}
\begin{rmk}
The free maps are precisely the cofibrations between quasi-free dg Lie algebras in the model category of positively graded dg Lie algebras defined by Quillen (see the remark after \cite[Proposition II.5.5]{quillen69}).
\end{rmk}

\begin{dfn}
Let $\iota\colon\L_A\to \L_X$ be a free map of quasi-free dg Lie algebras. 
We say that an endomorphism $\varphi$ of $\L_X$ is  \textit{$\iota$-relative}  if the diagram
\begin{equation*}
\xymatrix{&\L_A\ar[dl]_{\iota}\ar[dr]^{\iota}&\\
\L_{X}\ar[rr]_\varphi&&\L_{X}}
\end{equation*}
commutes strictly. 
We say that two $\iota$-relative endomorphisms $\varphi$ and $\psi$ are \textit{$\iota$-equivalent} if there exists a homotopy $h\colon \L_X\to \L_X\ot \Lambda(t,dt)$ from $\varphi$ to $\psi$ that preserves $\L_A$ in the following sense: $h(y)= y\ot1$ for every $y\in \L_A$ (cf. \cite[§ 14 (a)]{felixrht}). In this case we can also say that $\varphi$ and $\psi$ are homotopic relative $\L_A$, and is denoted by $\varphi\simeq_{\rel \L_A} \psi$. As in \cite[§ II.5]{tanre83}, one proves that the relation of being homotopic relative to $\L_A$ is an equivalence relation such that if $\varphi_1\simeq_{\rel \L_A} \psi_1$ and $\varphi_2\simeq_{\rel \L_A} \psi_2$, then  $\varphi_2\circ\varphi_1\simeq_{\rel \L_A}\psi_2\circ\psi_1$.
\end{dfn}

\begin{prop}\label{prop:groupofcomponents}
Let $\iota\colon\L_A\to \L_X$ be a free map that models a based cofibration  $A\subset X$ of simply connected spaces. Then the monoid $\pi_0(\End_{A_\Q}(X_\Q))$ is isomorphic to the monoid
\begin{equation}\label{hu}\{\varphi\colon\L_X\to\L_X\ |\ \varphi \text{ is $\iota$-relative}\}/\iota\text{-equivalence}\end{equation}
\end{prop}

\begin{proof}
We have that the restriction map $\End(\L_X)_\bullet\to \mathrm{Map}(\L_A,\L_X)_\bullet$ is a Kan fibration with fiber above $\iota$,  denoted by $\mathrm{End}_{\L_A}(\L_X)_\bullet$,  given by the pullback of $\Delta^0_\bullet\to \mathrm{Map}(\L_A,\L_X)_\bullet \leftarrow \End(\L_X)_\bullet$, where $\Delta^0_0$ is sent to $\iota$. As a simplicial set we have that 
$$\End_{\L_A}(\L_X)_\bullet = \{f_\bullet \in\mathrm{Map}(\L_X,\L_X\ot \Omega_\bullet)\ |\ f_n(y) = y\ot 1,\ \forall y\in\L_A\}$$
This together with Lemma \ref{mk}, implies that
$\End_{\L_A}(\L_X)_\bullet$ is weakly equivalent to $\End_A(X)$ as a monoid. In particular, they have isomorphic monoids of components. It is straightforward to show that  $\pi_0(\End_{\L_A}(\L_X)_\bullet)$ is given by the set in \eqref{hu}.
\end{proof}

\begin{cor}\label{cor:pi-noll}
Let $\iota\colon\L(V)\to \L(V\oplus W)$ be a minimal relative dg Lie model  for a based cofibration  $A\subset X$ of simply connected spaces. Then the group $\pi_0(\aut_{A_\Q}(X_\Q))$ is isomorphic to the group $$\{\varphi\colon\L(V\oplus W)\to\L(V\oplus W)\ |\ \varphi \text{ is a $\iota$-relative automorphism}\}/\iota\text{-equivalence}.$$
\end{cor}
\begin{proof}
Since a quasi-isomorphism $\L(V\oplus W)\to \L(V\oplus W)$ that preserves $\L(V)$ elementwise is an automorphism (Theorem \ref{thm:hautOfMinimal}), the assertion follows as a consequence of Proposition \ref{prop:groupofcomponents}. 
\end{proof}

\begin{lemma}\label{lemma:finiteDim}
Let $\L(V)\to \L(V\oplus W)$ be a minimal relative  dg Lie model for a based cofibration $A\subset X$ of simply connected spaces, where $\L(V)$ is a minimal model for $A$. If both $H^*(A;\Q)$ and $H^*(X;\Q)$ are finite dimensional, then $V$ and $W$ are finite dimensional. 
\end{lemma}

\begin{proof}
Since $\L(V)$ is a minimal model for a space $A$, it follows that $V\cong s^{-1}\tilde H_*(A)$ (see e.g. \cite[§ 24 (b)]{felixrht}), which gives that $\dim(V)<\infty$. Moreover, due to the Quillen equivalence between the model category of simply connected dg Lie algebras and the model category of simply connected rational pointed topological spaces, we get that the homotopy (co)limit of a diagram of Lie models, models the homotopy (co)limit of their geometric realizations. In particular $(\L(W),\bar d)$, with the  differential induced from $(\L(V\oplus W),d)$, is a minimal model for $X/A$, which has finite dimensional homology. This gives that $\dim(W)<\infty$.
\end{proof}

\begin{prop}\label{prop:autmorphismsOfLieAlg}
Given  a finitely generated, simply connected, relative dg Lie algebra $\L(V\oplus W)$ with base $\L(V)$, the group of automorphisms of $\L(V\oplus W)$ that fix the dg subalgebra $\L(V)$ pointwise, denoted by $\Aut_{\L(V)}(\L(V\oplus W))$, is isomorphic to the group of $\Q$-points of a linear algebraic group over $\Q$, denoted by $\Aaut_{\L(V)}(\L(V\oplus W))$
\end{prop}

\begin{proof}
Let $\{v_1,\dots,v_a\}$ and $\{w_1,\dots,w_b\}$ be bases for $V$ and $W$ respectively. Let $$m = \max(|v_1|,\dots,|v_a|,|w_1|,\dots,|w_b|),$$ and let $\Sigma= \L(V\oplus W)_{\leq m}$. Note that $\Sigma$ is closed under the differential and thus the restriction of the differential to $\Sigma$ defines a linear operator $D\colon \Sigma\to \Sigma$. Moreover, note that if $[a,b]\in \Sigma$ then $a,b\in \Sigma$.  Let $\beta =\{x_1,\dots,x_r, x_{r+1},\dots ,x_{s}\}$ be a basis for $\Sigma$ where $x_i$ is of homogeneous degree and where $\{x_1,\dots, x_r\}$ is a basis for $\L(V)_{\leq m}\subseteq \Sigma$. Let $\Aaut_{\L(V)}(\L(V\oplus W))$ be the algebraic subgroup of $\GL(\Sigma)$ 
that is defined by the automorphisms of $\Sigma$ that are (i) degree preserving, (ii) commutes with the linear operator $D$, (iii) satisfies $\varphi[a,b] = [\varphi(a),\varphi(b)]$, (iv) $\varphi(x_i) =x_i $ for $1\leq i \leq r$. Note that conditions (i)-(iv) are all algebraic, which is easily seen when considering $\varphi$ as a matrix relative to the basis $\beta$. Since $\Sigma$ is finite dimensional $\Q$-vector space, it follows that  $\Aaut_{\L(V)}(\L(V\oplus W))$ is a linear algebraic group over $\Q$.
\end{proof}

\begin{dfn} Let $L$ be a dg Lie algebra. A \textit{derivation} $\theta$ on $L$ is a linear map $\theta\colon L\to L$ that satisfies $\theta([a,b]) = [\theta(a),b]+ (-1)^{|a||\theta|}[a,\theta(b)]$. The space of derivations on  $L$ is denoted by $\Der(L)$ and has a dg Lie algebra structure where the Lie bracket and the differential $D$ is given by
$$
[\theta,\eta] = \theta \circ \eta - (-1)^{|\theta||\eta|}\eta\circ\theta
\qquad\text{and}\qquad
D(\theta) = [d,\theta].
$$
where $d$ is the differential of $L$.
\end{dfn}

\begin{lemma}\label{lemma:tgtSpaceOfAutomorphismsIsDerivations}
Given  a finitely generated, simply connected, relative dg Lie algebra $\L(V\oplus W)$ with base $\L(V)$, the Lie algebra of the linear algebraic group $\Aaut_{\L(V)}(\L(V\oplus W))$ is given by the zero cycles of the dg Lie algebra
$$\Der(\L(V\oplus W)\|\L(V)):= \{\delta\in \Der(\L(V\oplus W))\,|\,\delta(x)=0, \, \forall x\in \L(V)\}.
$$
\end{lemma}

\begin{proof}
Set $\Q[\epsilon] = \Q[Y]/Y^2$. The Lie algebra $L(\mathcal G)$ of an algebraic group $\mathcal G$ over $\Q$, is given by the tangent space of $\mathcal G$ at the identity element $\id\in\mathcal G(\Q)$
$$\mathscr T_{\id}(\mathcal G) = \{ \phi\in \mathcal G(\Q[\epsilon])\ |\ \phi \equiv \id (\mathrm{mod}\,\epsilon)\}.$$
Hence, the Lie algebra of $\mathcal G=\Aaut_{\L(V)}(\L(V\oplus W))$ is given by 
$$\{\phi \in \Aut_{\L(V)\ot\Q[\epsilon]}(\L(V\oplus W)\otimes \Q[\epsilon])\ |\   \phi \equiv \id (\mathrm{mod}\,\epsilon) \}$$

Given some $\phi\in \mathscr T_{\id}(\mathcal G)$, there is an associated function $\delta\colon \L(V\oplus W)\to \L(V\oplus W)$ given by $$\phi(x+\epsilon x') = \phi(x)+\epsilon \phi(x')= x + \epsilon \delta(x)+\epsilon x'$$
for $x+\epsilon x'\in \L(V\oplus W)\ot \Q[\epsilon]$.
We have that $\phi$ commutes with the  bracket and the differential if and only if $\delta$ is a derivation of degree zero that commutes with the differential. I.e.
$$\delta\in Z_0(\Der(\L(V\oplus W))).$$
Moreover, $\phi$ fixes $y+\epsilon y'\in \L(V)\oplus \Q[\epsilon]$ if and only if $\delta(y)=0$, i.e. $\delta\in Z_0(\Der(\L(V\oplus W)\|\L(V)))$. Thus, any derivation with that property yields an element of $\mathscr T_{\id}(\mathcal G)$. In particular, we have that $
\mathscr T_{\id}(\mathcal G) \cong Z_0(\Der(\L(V\oplus W)\|\L(V)))$.
\end{proof}

Now we are ready to prove the first part of Theorem \ref{thm:B}.
\begin{thm}[\text{cf. \cite[Theorem 3.4]{BlockLazarev}}]\label{thm:alggrp}
Let $A\subset X$ be a based cofibration of simply connected spaces of  the homotopy types of finite CW-complexes. Then $\pi_0(\aut_{A_\Q}(X_\Q))$ is isomorphic to the group of $\Q$-points of a linear algebraic group $\Aaut_{\L(V)}^h(\L(V\oplus W))$.
\end{thm}

\begin{proof} Let $\L(V)\to \L(V\oplus W)$ be a minimal model for the pointed cofibratin $A\subset X$. By Lemma \ref{lemma:finiteDim}, $V$ and $W$ are finite dimensional and thus $\Aaut_{\L(V)}\L(V\oplus W)$ is an algebraic group. Let $\g = Z_0(\Der(\L(V\oplus W)\|\L(V)))$ denote the Lie algebra of $\Aut_{\L(V)}\L(V\oplus W)$ (Lemma \ref{lemma:tgtSpaceOfAutomorphismsIsDerivations}).
It follows  by Corollary \ref{cor:pi-noll} that $\pi_0(\aut_{A_\Q}(X_\Q))$ is isomorphic to $\Aut^h_{\L(V)}\L(V\oplus W):=\Aut_{\L(V)}\L(V\oplus W)/K$ where
 $$K=\{f\in \Aut_{\L(V)}\L(V\oplus W)\ |\ f\simeq_{\rel \L(V)} \id\}.$$

 We will show that $K$ is the group of $\Q$-points of a unipotent normal subgroup $\mathcal K\subset \Aaut_{\L(V)}(\L(V\oplus W))$, and hence, by \cite[Proof of Theorem 9.5]{wilkerson76} we have that $\Aut_{\L(V)}(\L(V\oplus W))/K$ are the $\Q$-points of the quotient  $\Aaut_{\L(V)}(\L(V\oplus W))/\mathcal K$ which we will denote by $\Aaut^h_{\L(V)}(\L(V\oplus W))$.

We proceed as in the  proof of \cite[Theorem 3.4]{BlockLazarev} by setting $\h = B_0(\Der(\L(V\oplus W)\|\L(V)))$. In particular, an element of $\h$ is of the form $[G,d] = G\circ d + d\circ G$, where $G$ is a derivation of degree 1 that vanishes on $\L(V)$. Thus, elements of $\h$ vanish on $\L(V)$, so by the relative minimality of $\L(V\oplus W)$, we have that  elements of $\h$ strictly raises the word length of elements of $\L(V\oplus W)$  (or sends them to 0). Since $\L(V\oplus W)$ is concentrated in positive degrees, it follows that $\h$ is nilpotent.  Thus, since $\h$ is also finite dimensional, it follows that $\exp(\h)$ corresponds to a unipotent linear algebraic subgroup of $\Aaut_{\L(V)}(\L(V\oplus W))$ (\cite[Theorem 14.37]{milne2017}). Moreover, we have that the Lie algebra of  $\mathcal K$, $\h = B_0(\Der(\L(V\oplus W)\|\L(V)))$ is an ideal in the Lie algebra $\g=Z_0(\Der(\L(V\oplus W)\|\L(V)))$ of $\Aaut_{\L(V)}\L(V\oplus W)$, which gives that $\mathcal K$ is normal in  $\Aaut_{\L(V)}(\L(V\oplus W))$ (\cite[§ 13.3]{humphreys}).
\end{proof}

For $A_\Q=*$, we get by Theorem \ref{thm:alggrp} that $\pi_0(\aut_*(X_\Q))$ is isomorphic to the group of $\Q$-points of  a linear algebraic group $\Aaut^h(\L W)$ where $\L W$ is a minimal dg Lie algebra model for $X$. Since $X_\Q$ is simply connected, we have that $\pi_0(\aut_*(X_\Q))\cong \pi_0(\aut(X_\Q))$. Hence, it follows $\pi_0(\aut(X_\Q))$ can also be viewed as the group of $\Q$-points of the same linear algebraic group $\Aaut^h(\L W)$.

Similarly, Sullivan \cite{sullivan77} proved that $\pi_0(\aut(X_\Q))$ is also isomorphic to the $\Q$-points of a linear algebraic group expressed in terms of the minimal Sullivan model for $X_\Q$. If the homology of $X$ is concentrated in degrees $\leq n$ and $X(n)$ is $n$'th Postnikov stage of $X$, then $\pi_0(\aut(X)) \cong \pi_0(\aut(X(n)))$. If $\Lambda V$ is the  minimal Sullivan model for $X$, then $\Lambda V^{\leq n}$ is the minimal model for $X(n)$, which is finitely generated, and thus yielding an algebraic group $\Aaut^h(\Lambda V^{\leq n})$. By similar arguments to the ones in the proof of Theorem \ref{thm:alggrp}, it follows that $\pi_0(\aut(X_\Q(n)) \cong \pi_0(\aut(X_\Q))$ is isomorphic the $\Q$-points of the algebraic group $\Aaut^h(\Lambda V^{\leq n})$. A priori, the algebraic groups $\Aaut^h(\Lambda V^{\leq n})$ and $\Aaut^h(\L W)$ are different, but by the  main result of \cite{salehAG}, they are isomorphic.

\begin{thm}[\cite{salehAG}]\label{thm:salehAGG}
Let $X$ be a simply connected $n$-dimensional finite CW-complex with a minimal Sullivan model $\Lambda V$ and a minimal dg Lie algebra model $\L W$.
There is an isomorphism of algebraic groups,
$$\Aaut^h(\L W)\cong \Aaut^h(\Lambda V^{\leq n}).$$
\end{thm}

\begin{thm}\label{thm:AlgMap} 
Let $A \subset X$ be a based cofibration of simply
connected spaces of the homotopy types of finite CW-complexes.
The map $\pi_0(\aut_{A_\Q}(X_\Q))\to \pi_0(\aut(X_\Q))$ induced by the inclusion $\aut_{A_\Q}(X_\Q)\hookrightarrow\aut(X_\Q)$ is modelled by the  $\Q$-points of a morphism of algebraic groups $$\Aaut^h_{\L(V)}(\L(V\oplus U))\to \Aaut^h(\L(U)).$$
\end{thm}
\begin{proof}
Let $\L(V)\to \L(V\oplus W)$ be a relative minimal dg Lie algebra model for $A\subset X$ and let $\theta \colon \L(U)\xrightarrow\sim \L(V\oplus W)$ be a minimal model for $\L(V\oplus W)$. Note that $\L(V\oplus W)$ is finitely generated

The dg Lie algebra $\L(U)$ is  fibrant-cofibrant so there exists a homotopy inverse  $\mu\colon \L(V\oplus W)\xrightarrow\sim  \L(U)$. Given any automorphism $f\in\Aut_{\L(V)}(\L(V\oplus U))$, we have that $\mu\circ f\circ\theta\colon \L(U)\to \L(U)$ is a quasi-isomorphism, and thus an isomorphism by minimality. 
This gives rise to a map of varieties (not necessary of algebraic groups) $\xi\colon\Aaut_{\L(V)}(\L(V\oplus W))\to \Aaut(\L(U))$ given by $f\mapsto \mu\circ f\circ\theta$ (recall by Proposition \ref{prop:autmorphismsOfLieAlg}) that $\Aaut(\L(U))$ and $\Aaut_{\L(V)}(\L(V\oplus W))$ can be viewed as algebraic subgroups of
$\GL(\L(U)_{\leq m})$ and $\GL(\L(V\oplus W)_{\leq m})$ respectively, where 
$m$ is maximal homological degree in $V\oplus W$, and thus $\theta$ and $\mu$ may be represented by finite matrices).

We want to show that $\xi(\exp(B_0(\Der(\L(V\oplus W)\|\L(V)))))\subseteq \exp(B_0(\Der(\L(U))))$. Since a morphism of varieties over a field $k$ is completely determined by the induced morphisms of $\overline k$-points, where $\overline k$ is an algebraic closure of $k$, it is enough to show that $\xi$ induces an inclusion on the $\C$-points. An element $\varphi\in\exp(B_0(\Der(\L(V\oplus W)\|\L(V))))(\C)$ is a relative automorphism of $\L(V\oplus W)\ot \C$ that is relative homotopic to the identity. Since $\theta$ and  $\mu$ are homotopy inverses to each other, it follows that 
$\mu_\C\circ\varphi\circ\theta_\C$ is homotopic to the identity on $\L(U)\ot\C$, which is thus an element of $\exp(B_0(\Der(\L(U))))(\C)$

Thus $\xi(\exp(B_0(\Der(\L(V\oplus W)\|\L(V)))))\subseteq \exp(B_0(\Der(\L(U))))$,  so $\xi$ induces a map of varieties 
$$
\begin{array}{ccc}
\bar\xi\colon \Aaut_{\L(V)}(\L(V\oplus W))/\exp(I) & \longrightarrow &  \Aaut(\L(U))/\exp(J),
\end{array}
$$
where $I= B_0(\Der(\L(V\oplus W)\|\L(V))))$ and $J=B_0(\Der(\L(U)))$.
This map respects the group structure on the $\C$-points is therefore a  map of algebraic groups (see e.g. \cite[Lemma 3.12]{salehAG}). The $\Q$-points of $\bar \xi$ induces a morphism of groups
$$\bar\xi(\Q)\colon\Aut^h_{\L(V)}(\L(V\oplus W))\to \Aut^h(\L(U))$$
that models the desired map.
\end{proof}

\section{Finite presentation}\label{sec:finitePres}
Throughout this and next section, $A\subset X$ is a cofibration of simply connected spaces of the homotopy type of finite CW-complexes.

We start by setting some notation. In order to simplify the notation, $ \pi_0(\aut(X))$, $\pi_0(\aut_A(X))$, $\pi_0(\aut(X_\Q))$ and $\pi_0(\aut_{A_\Q}(X_\Q))$ are denoted by $\E(X)$, $\E_A(X)$, $\E(X_\Q)$ and $\E_{A_\Q}(X_\Q)$, respectively. The rationalization functor induces a group homomorphism
$$r\colon \E(X)\to \E(X_\Q)
$$
and the inclusion $\aut_A(X)\hookrightarrow \aut(X)$ induces a group homomorphism
$$
j\colon \E_A(X)\to \E(X).
$$ 
The rational analogue gives a map
$$
j_\Q\colon \E_{A_\Q}(X_\Q)\to \E(X_\Q).
$$

Now we recall some of the theory of arithmetic groups needed in this section. 
\begin{dfn} Given a linear algebraic group $\mathcal G$ over $\Q$, there exists (by definition) a faithful representation $\rho\colon \mathcal G(\Q)\to \GL(V)$, where $V$ is a finite dimensional $\Q$-vector space. Given a lattice $L\subset V$, let 
$$\mathcal G_L = \{g\in \mathcal G(\Q)\ |\ \rho(g)(L)=L\}.$$
We say that a subgroup $\Gamma\subseteq \mathcal G$ is an \textit{arithmetic} subgroup of $\mathcal G$ if and only if $\Gamma$ and $\mathcal G_L$ are commensurable in $\mathcal G$. The definition of arithmetic groups is independent of choice of $\rho$ and $L$.\end{dfn}

\begin{lemma}[\text{\cite[Theorem 5.14]{milne2017}}]\label{lemma:imagealggroup}
Let $\varphi\colon \mathcal G\to \mathcal G'$ be a homomorphism of linear algebraic groups. Then $\varphi(\mathcal G)$ is a linear algebraic subgroup of $\mathcal G'$.
\end{lemma}
\begin{lemma}\label{lemma:ArithGrpsUnderHomomorph}
 Let $\mathcal G$ be an algebraic group and $\Gamma\subseteq \mathcal G(\Q)$ be an arithmetic subgroup. Given any algebraic subgroup $\mathcal H\subset \mathcal G$, the group $\Gamma\cap \mathcal H(\Q)$ is an arithmetic subgroup of $\mathcal H$.
\end{lemma}

\begin{proof}  For  $\Gamma = \mathcal G_L$ the assertion is certainly true since $\mathcal G_L\cap H = H_L$. It is straightforward to show that the commensurability relation is preserved under intersection with subgroups, i.e. if $\Gamma_1,\Gamma_2\subseteq \mathcal G$ are commensurable in $\mathcal G$, then  $\Gamma_1\cap \mathcal H, \Gamma_2\cap \mathcal H\subset  \mathcal H$ are also commensurable in $\mathcal H$.
\end{proof}

We recall that arithmetic groups are finitely presented by \cite[§ 5]{borel62}.

\begin{prop}\label{prop:finitepresentedd}
$j(\E_A(X))$ is commensurable with an arithmetic group. In particular $j(\E_A(X))$ is finitely presented.
\end{prop}

\begin{proof}
By combining Lemma \ref{lemma:imagealggroup} and Theorem \ref{thm:AlgMap} we get that $j_\Q(\E_{A_\Q}(X_\Q))$ is the group of $\Q$-points of the algebraic subgroup $\mathcal H\subseteq \Aaut^h(\L(U))$, where $\L(U)$ is the minimal dg Lie algebra model for $X$. 
Set 
$$
G_A(X) = r^{-1}(j_\Q(\E_{A_\Q}(X_\Q)))\subseteq \E(X).
$$
We have by definition that $r(G_A(X)) =  r(\E(X))\cap j_\Q(\E_{A_\Q}(X_\Q)) = r(\E(X))\cap \mathcal H(\Q)$. The classical Sullivan-Wilkerson theorem combined with Theorem \ref{thm:salehAGG} gives that  $r(\E(X))$ is an arithmetic subgroup of $\Aaut^h(\L(U))$  and hence $r(G_A(X))$ is an arithmetic group of $\mathcal H$ by Lemma \ref{lemma:ArithGrpsUnderHomomorph}. Moreover, since $r$ is finite to one  (\cite[Theorem 10.2 (i)]{sullivan77}), it follows that $G_A(X)$ is commensurable up to finite kernel with an arithmetic group.

Note that $f\in G_A(X)$ if and only if $f_\Q|_{A_\Q}$ is homotopic to the inclusion of $A_\Q$ into $X_\Q$. Note that $j(\E_A(X)))\subseteq G_A(X)$, and hence it is enough to show that $j(\E_A(X)))$ has finite index in $G_A(X)$  in order to deduce that $j(\E_A(X)))$ is commensurable up to finite kernel with an arithmetic group. We have by \cite[Theorem 10.2 (i)]{sullivan77} that $r\colon [A,X]\to [A_\Q,X_\Q]$ is finite-to-one. If the cardinality of the inverse image of the inclusion $A_\Q\hookrightarrow X_\Q$ is $k\in \Z_{\geq 1}$, then the cardinality of $G_A(X)/j(\E_A(X))$ is at most $k$, which proves the commensurability condition. Since arithmetic groups are finitely presented and finite presentability is stable under extensions (\cite[§ V.A.15]{harpe}), it follows  that $j(\E_A(X))$ is finitely presented.
\end{proof}

\begin{thm}\label{thm:finitePresentation}
$\E_A(X)$ is finitely presented.
\end{thm}
\begin{proof}
The fibration $\aut_A(X)\to \aut(X)\to \map(A,X)$, yields  an exact sequence 
$$\pi_1(\map(A,X),\incl)\xrightarrow\partial\E_A(X)\xrightarrow j \E(X).$$
We have that any component of $\map(A,X)$ is a nilpotent space (\cite[Chapter II, Theorem 2.5. (ii)]{hmr75}), implying that  $\pi_1(\map(A,X),\incl)$ is a nilpotent group, and thus that $\pi_1(\map(A,X),\incl)/\ker(\partial)$ is a nilpotent group.

By Federer \cite{federer}, given two CW-complexes $A$ and $X$ and a continuous map $f\colon A\to X$ where $A$ is finite and $X$ is simple, there is a spectral sequence
$E^2_{p,q} =H^p(A,\pi_q(X))\Rightarrow \pi_{q-p}(\map(A,X),f)$.
Since all the groups $E^2_{p,q}$ are abelian and vanishes for $p$ greater than the dimension of $A$, it follows that $\pi_1(\map(A,X),f)$ is a solvable group. In particular, there is a  subnormal series
$$
1=G_1\triangleleft G_2\triangleleft \cdots\triangleleft G_k =\pi_1(\map(A,X),f)
$$
where $G_i/G_{i-1} = E^\infty_{i-1,i}$. Since $E^2_{p,q}$ is abelian and finitely generated and $E^\infty_{p,q}$ is a subquotient of it, it follows that $E^\infty_{p,q}$ is finitely generated. Now one proves that $G_i$ is finitely generated for all $i$ by induction on $i$. For $i=1$ this is trivial. If $G_{i-1}$ is finitely generated and $G_i/G_{i-1}$ is finitely generated it follows that $G_i$ is finitely generated (the union of the set of generators of $G_{i-1}$ and a set of elements of $G_i$ that projects to the generators of $G_i/G_{i-1}$ generates all of $G_i$).

In particular $\pi_1(\map(A,X),\incl)$ is finitely generated.

Let $G$ denote the quotient $\pi_1(\map(A,X),\incl)/\ker(\partial)$, which is again a finitely generated nilpotent group.  Finitely generated nilpotent groups are finitely presented \cite[§ V.A.19]{harpe} and thus $G$ is finitely presented.

We get a short exact sequence 
$$1\to G\to \E_A(X)\to j(\E_A(X))\to 1$$ where $G$ and $j(\E_A(X))$ are finitely presented (Proposition \ref{prop:finitepresentedd}). Now it follows that $\E_A(X)$ is finite presented since finite presentability is stable by extensions (\cite[§ V.A.15]{harpe}).
\end{proof}

\section{Finite Postnikov stages}\label{sec:finite-postnikov}
\begin{dfn}
Given fixed inclusions $A\subset X$ and $A\subset Y$, let $\map_A(X,Y)$ denote the space of continuous maps from $X$ to $Y$ that preserve $A$ pointwise, and let $[X,Y]_A = \pi_0(\map_A(X,Y))$. 
\end{dfn}
\begin{rmk}
The set $[X,X]_A$ is a monoid under composition.
\end{rmk}

Given an $n$-dimensional CW-complex $X$, one can give the $n$-Postnikov stage $X(n)$ a CW-complex structure, where $X$ is a subcomplex and where $X(n)/X$ has cells of dimension $\geq n+2$. In particular, there are no  $(n+1)$-dimensional cells in $X(n)$ and $X$ is thus also the $(n+1)$-skeleton of $X(n)$.

\begin{lemma}\label{lemma:last}
Let $X$ be an $n$-dimensional simple CW-complex, and let $A\subset X$ be a subcomplex.  There is an isomorphism of monoids $[X(n),X(n)]_A \cong [X,X]_A$.
\end{lemma}

\begin{proof}
We start by defining a map $\alpha\colon [X(n),X(n)]_A\to [X,X]_A$ in the following way. Given a homotopy class $\bar f \in [X(n),X(n)]_A$, take a cellular representative $f\colon X(n)\to X(n)$ (this is possible due to the cellular approximation theorem). Since $X$ is the $n$-skeleton of $X(n)$ it follows that $f(X)\subset X$, and thus the restriction of $f$ to $X$ defines an endomorphism of $X$. Let $\alpha(\bar f)$ be the homotopy class of $f|_X\colon X\to X$.

We need to show that $\alpha$ is well-defined. Given two different cellular representatives $f$ and $g$ for $\bar f$, there is an $A$-preserving cellular homotopy 
$$
H\colon X(n)\times I\to X(n)
$$
from $f$ to $g$. Since $H$ is cellular, $X$ is $n$-dimensional and $X(n)$ has no $(n+1)$-dimensional cells  it follows that $H(X\times I)\subset X$. Hence the restriction of $H$ to $X\times I$ defines an $A$-preserving homotopy from $f|_X$ to $g|_X$. This proves that $\alpha$ is well-defined. Clearly $\alpha$ preserves compositions, and therefore a morphism of monoids.

Now we prove that $\alpha$ is injective. Assume that $\alpha(\bar f)= \alpha(\bar g)$. That means that there is an $A$-preserving homotopy $h\colon X\times I$ from $f|_X$ to $g|_X$. In particular there is a map  
$$
j\colon X(n)\times \{0\} \cup X\times I \cup X(n)\times \{1\}\to X(n) 
$$
where $j|_{X(n)\times \{0\}} =f$, $j|_{X\times I}= h$ and $j|_{X(n)\times \{1\}}= g$. There is an inclusion from the domain of $j$ to $X(n)\times I$. A lift $\tilde h$  of $j$ to $X(n)\times I$ is an $A$-preserving homotopy from $f$ to $g$. We have that the obstructions to such lifts live in
$$
H^{k+1}(X(n)\times I, X(n)\times \{0\} \cup X\times I \cup X(n)\times \{1\};\pi_k(X(n)))
$$ 
$$
\cong H^{k+1}(\Sigma(X(n)/X);\pi_k(X(n))) \cong H^k(X(n)/X;\pi_k(X(n)))
$$
(see \cite[Proposition 4.72]{hatcher}).

Since $X(n)/X$ only has cells of dimension $\geq n+2$, the homology groups vanish for $k\leq n+1$. Since $\pi_t(X(n))= 0$ for $t\geq n+1$, it follows that the homology groups  vanish for all $k\geq n+1$. Hence all obstructions vanish. Thus an $A$-relative homotopy from $f$ to $g$ exists and injectivity is proved.

For surjectivity, we will show that any endomorphism of $X$ extends to an endomorphism of $X(n)$. This is equivalent to the lifting problem
$$
\xymatrix{
X\ar[r]^f\ar@{>->}[d] & X\ar@{>->}[r]&X(n)\\
X(n)\ar@{-->}[rru]
}
$$
for every endomorphism $f\colon X\to X$. 
Obstructions to such lifts live in 
$$
H^{k+1}(X(n),X;\pi_k(X(n))) \cong H^{k+1}(X(n)/X;\pi_k(X(n))).
$$
As before, all these cohomology groups vanish, and thus lifts do always exist. This completes the proof.
\end{proof}

\begin{rmk}
    The proof mentioned above remains valid when replacing $n$ with any $m\geq n$.
\end{rmk}

\begin{cor} Under the assumptions of Lemma \ref{lemma:last}, there is an isomorphism of groups $\pi_0(\aut_A(X(m)))\cong \pi_0(\aut_A(X))$ for every $m\geq n$.
\end{cor}

\bibliographystyle{amsalpha}
\bibliography{references}

\providecommand{\bysame}{\leavevmode\hbox to3em{\hrulefill}\thinspace}
\providecommand{\MR}{\relax\ifhmode\unskip\space\fi MR }
\providecommand{\MRhref}[2]{%
  \href{http://www.ams.org/mathscinet-getitem?mr=#1}{#2}
}
\providecommand{\href}[2]{#2}
\begin{thebibliography}{KRW20}

\bibitem[Bau95]{baues95}
H-J. Baues, \emph{Homotopy types}, Handbook of Algebraic Topology (I.M. James,
  ed.), Elsevier, Amsterdam, 1995, pp.~1--72.

\bibitem[Ber20]{berglund17}
A.~Berglund, \emph{{Rational models for automorphisms of fiber bundles}}, Doc.
  Math. \textbf{25} (2020), 239--265.

\bibitem[BL05]{BlockLazarev}
J.~Block and A.~Lazarev, \emph{{André–Quillen cohomology and rational
  homotopy of function spaces}}, Adv. Math. \textbf{193} (2005), no.~1,
  18–39.

\bibitem[BM20]{BM14}
A.~Berglund and I.~Madsen, \emph{{Rational homotopy theory of automorphisms of
  manifolds}}, Acta Math. \textbf{224} (2020), 67--185.

\bibitem[Bor63]{borel62}
A.~Borel, \emph{Arithmetic properties of linear algebraic groups}, Proc.
  {I}nternat. {C}ongr. {M}athematicians ({S}tockholm, 1962), Inst.
  Mittag-Leffler, Djursholm, 1963, pp.~10--22. \MR{0175901}

\bibitem[Bre93]{bredonGT}
G.~E. Bredon, \emph{{Topology and Geometry}}, Graduate Texts in Mathematics,
  vol. 139, Springer-Verlag, New York, 1993.

\bibitem[BS20]{BS19}
A.~Berglund and B.~Saleh, \emph{{A dg Lie model for relative homotopy
  automorphisms}}, Homology Homotopy Appl. \textbf{22} (2020), 105--121.

\bibitem[CR19]{CiriciRoig}
J.~Cirici and A.~Roig, \emph{Sullivan minimal models of operad algebras}, Publ.
  Mat. \textbf{63} (2019), no.~1, 125–154.

\bibitem[dlH00]{harpe}
P.~de~la Harpe, \emph{Topics in geometric group theory}, Univ. of Chicago
  Press, 2000.

\bibitem[Far96]{farjoun}
E.~D. Farjoun, \emph{Cellular spaces, null spaces and homotopy localization},
  Lecture Notes in Mathematics, no. 1622, Springer-Verlag, Berlin, 1996.

\bibitem[Fed56]{federer}
H.~Federer, \emph{A study of function spaces by spectral sequences}, Trans.
  Amer. Math. Soc. \textbf{82} (1956), 340--361. \MR{79265}

\bibitem[FHT01]{felixrht}
Y.~Félix, S.~Halperin, and J-C. Thomas, \emph{{Rational Homotopy Theory}},
  Graduate Texts in Mathematics, vol. 205, Springer, 2001.

\bibitem[Gre19]{grey}
M.~Grey, \emph{On rational homological stability for block automorphisms of
  connected sums of products of spheres}, Algebr. Geom. Topol. \textbf{19}
  (2019), no.~7, 3359–3407.

\bibitem[Hat02]{hatcher}
A.~Hatcher, \emph{Algebraic topology}, Cambridge University Press, 2002.

\bibitem[Hin01]{hinich01}
V.~Hinich, \emph{Dg coalgebras as formal stacks}, J. Pure Appl. Algebra
  \textbf{162} (2001), no.~2-3, 209--250.

\bibitem[HMR75]{hmr75}
P.~Hilton, G.~Mislin, and J.~Roitberg, \emph{Localization of nilpotent groups
  and spaces}, North-Holland Pub. Co., 1975.

\bibitem[Hum75]{humphreys}
J.~E. Humphreys, \emph{Linear algebraic groups}, Graduate Texts in Mathematics,
  No. 21, Springer-Verlag, New York-Heidelberg, 1975. \MR{0396773}

\bibitem[KRW20]{K-RW}
M.~Krannich and O.~Randal-Williams, \emph{Mapping class groups of simply
  connected high-dimensional manifolds need not be arithmetic}, C. R. Math.
  Acad. Sci. Paris \textbf{358} (2020), no.~4, 469--473. \MR{4134256}

\bibitem[Kup22]{kupers22}
A.~Kupers, \emph{Mapping class groups of manifolds with boundary are of finite
  type}, arXiv preprint arXiv:2204.01945 (2022).

\bibitem[Mar96]{maruyama}
K.~Maruyama, \emph{Finitely presented subgroups of the self-homotopy
  equivalences group}, Math. Z. \textbf{221} (1996), no.~4, 537--548.
  \MR{1385167}

\bibitem[Mil17]{milne2017}
J.~S. Milne, \emph{Algebraic groups: The theory of group schemes of finite type
  over a field}, Cambridge Studies in Advanced Mathematics, vol. 170, Cambridge
  University Press, 2017.

\bibitem[Qui69]{quillen69}
D.~Quillen, \emph{{Rational homotopy theory}}, Ann. of Math. \textbf{90}
  (1969), no.~2, 205--295.

\bibitem[Sal23]{salehAG}
B.~Saleh, \emph{{Algebraic groups of homotopy classes of automorphisms and
  operadic Koszul duality}}, arXiv preprint arXiv:2305.02076 (2023).

\bibitem[Sch80]{scheerer80}
H.~Scheerer, \emph{Arithmeticity of groups of fibre homotopy equivalence
  classes}, Manuscripta Math. \textbf{31} (1980), no.~4, 413–424.

\bibitem[Sul77]{sullivan77}
D.~Sullivan, \emph{Infinitesimal computations in topology}, Inst. Hautes
  Études Sci. Publ. Math. (1977), no.~47, 269–331.

\bibitem[Tan83]{tanre83}
D.~Tanré, \emph{{Homotopie rationnelle: modèles de Chen, Quillen, Sullivan}},
  Lecture Notes in Mathematics, vol. 1025, Springer-Verlag, Berlin, 1983.

\bibitem[Wil76]{wilkerson76}
C.~W. Wilkerson, \emph{Applications of minimal simplicial groups}, Topology
  \textbf{15} (1976), no.~2, 111–130.

\bibitem[Wil80]{wilkersonerrata}
Clarence Wilkerson, \emph{Errata: ``{A}pplications of minimal simplicial
  groups'' [{T}opology {\bf 15} (1976), no. 2, 111--130; {MR} {\bf 53}
  \#6551]}, Topology \textbf{19} (1980), no.~1, 99. \MR{559480}

\end{thebibliography}
\noindent
\Addresses
\end{document}